\documentclass[10pt,twoside]{amsart}

\usepackage{amsmath,amssymb
}
\usepackage{amsfonts}
\usepackage{amsthm,amscd}
\usepackage{amsmath}
\usepackage{amssymb}
\usepackage{amstext}
\usepackage{color}
\usepackage{latexsym}
\usepackage{amscd,graphics}

\newcommand{\BB}{\mathcal{B}}
\newcommand{\CC}{\mathcal{C}}

\newcommand{\EE}{\mathcal{E}}

\newcommand{\LL}{\mathcal{L}}

\newcommand{\MM}{\mathcal{M}}
\newcommand{\NN}{\mathcal{N}}
\newcommand{\OO}{\mathcal{O}}

\newcommand{\RR}{\mathcal{R}}

\newcommand{\spe}{{\rm sp}}

\newcommand{\tr}{{\rm tr}\, }
\newcommand{\Det}{{\rm Det}\, }
\newcommand{\Fix}{{\rm Fix}\, }

\newcommand{\spec}{{\rm spec}}

\newcommand{\Erg}{{\rm Erg}}

 \DeclareMathOperator{\Id}{Id}

\newcommand{\real}{\mathbb{R}}
\newcommand{\complex}{\mathbb{C}}
\newcommand{\integer}{\mathbb{Z}}
\newcommand{\TTT}{\mathbb{T}}

\newcommand{\comments}[1]{} %Comment

\newtheorem{proposition}{Proposition}[section]

\newtheorem{theorem}[proposition]{Theorem}

\newtheorem{corollary}[proposition]{Corollary}

\theoremstyle{remark}
\newtheorem{remark}[proposition]{Remark}
\theoremstyle{definition}

\numberwithin{equation}{section}

\begin{document}

\title[Absence of Giulietti--Liverani resonances]{There are no deviations for the ergodic averages of the Giulietti--Liverani horocycle flows
on the two-torus}
\author{Viviane Baladi
}
\address{Laboratoire de Probabilit\'es, Statistique et Mod\'elisation (LPSM),  
CNRS, Sorbonne Universit\'e, Universit\'e de Paris,
4, Place Jussieu, 75005 Paris, France}
\email{baladi@lpsm.paris}

\date{\today}
\begin{abstract}
We show that the ergodic integrals for the  horocycle flow on the two-torus associated 
by Giulietti and Liverani to an Anosov diffeomorphism either grow linearly or are bounded, in other
words there are no
deviations. For this, we use topological invariance of the Artin--Mazur zeta function
to exclude  resonances outside of the open unit disc.  Transfer operators acting on suitable spaces of anisotropic distributions
and their Ruelle determinants are the key tools in the proof. As a bonus, we show that for any $C^\infty$ Anosov diffeomorphism $F$ on the two-torus, the
correlations for the measure of maximal entropy and $C^\infty$ observables
decay with a rate strictly smaller than $e^{-h_{top}(F)}$. We compare our results with
very recent related work of Forni.
\end{abstract}
\thanks{Most of this work was done while preparing and delivering
a minicourse on  Anisotropic spaces and applications to hyperbolic and parabolic dynamical systems (Oberwolfach, June 2019). VB is very grateful to MFO for offering this possibility.
	Thanks to Romain Dujardin, Giovanni Forni, S\'ebastien Gou\"ezel,  and 
	Carlangelo Liverani for useful discussions and to Paolo Giulietti for pointing out  typos. 
	Thanks also to the anonymous referee and the editor for
	comments which helped improve the presentation.
VB's research is supported
by the European Research Council (ERC) under the European Union's Horizon 2020 research and innovation programme (grant agreement No 787304).}
\maketitle

\section{Introduction}
%send to Liverani, Dujardin, Malo, Forni, Gouezel,/ Bandtlow/Rugh (Giulietti) Butterley, Flaminio,Filip, FRH
%Hasselblatt Katok memorial volume?

\subsection{The results  of Giulietti--Liverani}\label{sec0}

In a pioneering work, Giulietti and Liverani  \cite{GL} introduced a ``horocycle flow''
on the torus, renormalizable by a given Anosov diffeomorphism:
Fix $r>1$ and let $F: \TTT^2 \to \TTT^2$ be a $C^r$ Anosov diffeomorphism on the two-torus. By Franks--Newhouse the stable bundle $E^s$ of $F$
(see \cite[p. 805]{Hiraide}) is orientable. Fixing an orientation $E^s_+$
of $E^s$, Giulietti and Liverani assume that $DF$
preserves this orientation and they introduce the  flow $h^t=h_s^t$
on $\TTT^2$ obtained by solving $\partial_t h^t(f) = X (f)$, where $X(x)=v^s_{1,+}(x)$ is the unique vector
of $E^s_+(x)$ of unit  norm. We call $h^t$ the {\it (unit\footnote{Giulietti and Liverani actually consider more general $C^r$ time parametrizations
		but this is immaterial for the purposes of this work, as the resonances defined
		below are invariant under such time-changes since the transfer operator is modified
		by a $C^r$ coboundary \cite[(2.5)]{GL}.} speed)  Giulietti--Liverani 
(stable horocycle) flow (of $F$).}  (See Appendix~\ref{appA} for basic facts about such flows.)

For any continuous function $f:\TTT^2 \to \complex$, any $T>0$,
and any $x \in \TTT^2$, define the horocycle
integral
$
H_{x,T}(f)=\int_0^T f(h^t(x)) \, d t 
$.

By unique ergodicity, we have for any such $x$ and $f$
\begin{equation}\label{ergg}
\lim_{T\to \infty} \frac{H_{x,T}(f)}{T }
=\mu^s(f):= \int_{\TTT^2}  f(x) \, d \mu^s\, ,
\end{equation}
where $\mu^s$ is the unique invariant probability measure of the
flow $h^t$.

 Let
$h_{top}=h_{top}(F)>0$ be the topological entropy of
$F$. Giulietti and Liverani \cite{GL} find $r_0>1$ such that,
if $r\ge r_0$ there exists
a bounded linear operator $\widetilde \LL$  associated
with $F$
(see \eqref{deftildeL} below) acting on a Banach space $\widetilde \BB_{GL}$
 of distributions in a subset of the unit
tangent bundle of $\TTT^2$ such that the following holds:
The spectral radius of $\widetilde \LL$ on $\widetilde \BB_{GL}$ is $e^{h_{top}}$
and the essential
spectral radius of $\widetilde \LL$ is strictly smaller than $\tilde \rho_{GL}=\tilde \rho_{GL}(r)<1$.
In addition,   $\tilde \rho_0:=e^{h_{top}}$ is a simple eigenvalue and the only maximal eigenvalue. 
Restricting for simplicity to the case where $\widetilde \LL$ does not
have any eigenvalue of modulus one and, in addition, all of its
eigenvalues 
 of modulus larger than $\tilde \rho_{GL}$ have trivial Jordan blocks
(otherwise, additional factors $\log T^q$ with $q\ge 1$ can appear in the expansion
\eqref{GLexp} below),
we let  $\{\tilde \rho_j\}_{j=1}^{N_{GL}}$ be the $N_{GL}\ge 0$
eigenvalues of $\widetilde \LL$ of modulus 
$\tilde \rho_{GL}\le |\tilde \rho_j|<e^{h_{top}}$ (ordered by decreasing modulus
and repeated with multiplicity), 
and we let $\OO_j\in \widetilde \BB_{GL}^*$ be the associated eigenvectors
of the dual operator $\widetilde \LL^*$. We call $\{\tilde \rho_j\}_{j=0}^{N_{GL}}$ the {\it resonances}\footnote{We shall
see after \eqref{deftilderho} below that they do not depend on the Banach space.}
of $\widetilde \LL$. We call the $N_{1,GL}\le N_{GL}$ resonances
such that  $1\le |\tilde \rho_j |<e^{h_{top}}$ its {\it deviation resonances.}
By \cite{GL2} (see also \cite[Lemma 2.11]{GL}),  the maximal eigenvector
$\widetilde \LL^* \tilde \mu^s=e^{h_{top}} \tilde \mu^s$ satisfies $\pi^*(\tilde \mu^s) =\mu^s$
(this fixes a normalisation of $\tilde \mu^s$), where $\pi$ denotes the projection from the unit tangent bundle of $\TTT^2$ to $\TTT^2$.

The first main result
of Giulietti--Liverani  \cite[Theorem 2.8]{GL}  
gives  the following
expansion for the horocycle integral $H_{x,T}(f)$: For any 
$\delta \in (0, -\log \tilde \rho_{GL})$, there exists
 a constant $C<\infty$   such that,
for any $T>0$,
and any $x \in \TTT^2$ there exist real numbers $\{C_j(x,T)\}_{j=0}^{N_{GL}}$
 with 
$\sup_{x,T,j} |C_j(x,T)| \le C$,
such that for any $C^r$ function $f:\TTT^2 \to \complex$, we have\footnote{Giulietti--Liverani \cite{GL} limit the expansion to $|\tilde \rho_j| \ge 1$, i.e. $\theta_j \ge 0$, bounding the error term by
$C \|f\|_{C^{r}}$. They write $\ell_j(x,T,f)$
instead of $C_j(x,T)\cdot \OO_j (f\circ \pi)$. Their proof
gives the slightly more precise result \eqref{GLexp}.}
\begin{equation}\label{GLexp}
H_{x,T}(f)=T \cdot C_0(x,T) \cdot \mu^s(f) +
\sum_{j=1}^{N_{GL}} T^{\theta_j} \cdot C_j(x,T) 
\cdot \OO_j(f\circ \pi)
+ \RR_{x,T}(f)\, , 
\end{equation}
where 
$
\theta_j = \frac{ \log |\tilde \rho_j|}{h_{top}} \in (-\infty , 1) 
$
and, for any $T>0$,
$$  \sup_{x} |\RR_{x,T}(f)| \le C \bigl (T^{\theta_{min}} \|f\|_{C^{r}}+ \sup |f|\bigr )\, , 
\mbox{ where } \theta_{min}= \frac{\log \tilde \rho_{GL}+\delta}{h_{top}}<0\, .
$$
Decomposing $f=\mu^s(f)+f-\mu^s(f)$,  we get
\begin{equation}\label{dominant}
T \cdot C_0(x,T)=H_{x,T}(1)=\int_0^T 1\, dt=T \, ,
\end{equation}
up to replacing $\OO_j(f\circ \pi)$
by $\OO_j((f-\mu^s(f))\circ \pi)$ in the right-hand-side of \eqref{GLexp}.

We refer to the introduction of \cite{GL} for the motivation
arising from the well-known  works
of Forni and Flaminio giving expansions of ergodic
averages along horocycles of (e.g.) geodesic flows on compact surfaces of constant
negative curvature in terms of  eigenvalues of the Laplacian and of
horocycle-flow invariant distributions (see e.g. \cite{FF}).
Some of the results of Flaminio and Forni have been  extended to
variable negative curvature geodesic flows \cite{Ad} by adapting the ideas
in \cite{GL}.

\begin{remark}[Bound for $r_0$]\label{rr0}
	An upper bound for $r_0$ can be read off from the bound
	for $\tilde \rho_{GL}(r)$ in
	\cite{GL2}, noting that the limit is $r$ and not $r-1$ in the
	case $\iota=1$ there, and taking care of the fact that one of the regularity
	exponents is an integer. In  \cite{GL}, the authors mention that  choosing
	$r_0\ge 1$ large enough such that 
	$$e^{h_{top}(F)}/\lambda^{r_0/2}<1$$
	suffices
	(where $1< \lambda <\min  \{\lambda_u, 1/\lambda_s\}$, see \eqref{deflambda},
	and $\lambda_{u}$, $\lambda_s$ are defined in the beginning of \S\ref{noref}).
	Compare this with \eqref{??} below.
\end{remark}

\subsection{Absence of deviation resonances and two consequences}

The main technical result of the present work, Theorem~\ref{nores}, is that,
if $r\ge r_1$ is large enough ($r_1$ depends on the expansion and contraction
factors of $F$, see \eqref{??} and \eqref{???})
then $|\tilde \rho_j|<1$ for all
$j=1, \ldots, N_{GL}$.
(In other\footnote{``Men han har jo ikke noget paa,'' sagde et lille Barn.} words, {\it the transfer operator $\widetilde \LL$ does not have deviation resonances.})
As a consequence (Corollary~\ref{cor1}), if
$r\ge \max \{r_0, r_1\}$, the exponents $\theta_j$ in the expansion \eqref{GLexp}
all satisfy
$
\theta_j <0
$. 
In particular (even in the presence of Jordan blocks),  there exists  $C'<\infty$  and $\theta'_{min}<0$ such that
 \begin{equation}\label{GLexp'}
\sup_x |H_{x,T}(f)-T  \cdot \mu^s(f) |\le  
C' (T^{\theta'_{min}} \|f\|_{C^{r}}+ \sup |f|)\, ,  \, \forall \, T >0 \, .
\end{equation}

\noindent That is, {\it there are no deviations to the convergence of horocycle ergodic integrals.}

We next mention a  consequence of Theorem~\ref{nores} about the ``cohomological
equation.''
Since $h^t(x)$ is minimal,  Gottschalk--Hedlund's theorem (see \cite{GH} for a recent account) implies,  for any
continuous $f$, that $\sup_{x,T} |H_{x,T}(f)|<\infty$ if and only
if $f$ is a continuous coboundary, i.e. there exists a continuous function $\bar f$ 
such that for all $x$ and
all $T>0$, the following cohomological equation holds:
$$
\bar f(x)- \bar f\circ h^T(x)=\int_0^T f \circ h^t(x)\, dt \, .
$$

The expansion \eqref{GLexp'} from Corollary~\ref{cor1} thus implies the following {\it dichotomy} for the Giulietti--Liverani flow of  a $C^r$ Anosov diffeomorphism
with $r \ge \max\{r_0, r_1\}$: 
{\it If $f\in C^r$ then either $\mu^s(f)\ne 0$ (and $H_{x,T}(f)$
grows linearly) or $f$ is a continuous coboundary.}

\medskip

 Another immediate consequence of Theorem~\ref{nores} about the absence of
 deviation resonances for $\widetilde \LL$   is Corollary~ \ref{cor2}: For any $C^\infty$ Anosov diffeomorphism $F$ on the two-torus, the
{\it correlations for the unique measure of maximal entropy $\mu_{top}$  and $C^\infty$ observables
decay with rate strictly smaller than $e^{-h_{top}(F)}$.} 
(In fact Corollary~\ref{cor2} only requires $C^r$ for large
enough $r$.)
\medskip

Giovanni Forni \cite{F} obtained (independently and
simultaneously) related results\footnote{We focus on the analogue
of Corollary~\ref{cor2} but Forni's main result is an analogue
of Corollary~\ref{cor1} about equidistribution of stable or unstable curves.}, in a
more general setting: For any $C^r$ ($r>1$) pseudo Anosov diffeomorphism on a compact surface, the correlation spectrum of the Margulis measure is
determined by the action of the diffeomorphism on the first cohomology, up
to a power law error term (his results do not
imply Corollary~\ref{cor2}, even when the diffeomorphism is Anosov and $r$ is large,
see Remark~\ref{compareGF}).

\smallskip

We end this introduction by mentioning  that if $F$
is a $C^r$ Anosov diffeomorphism on a compact connected $C^\infty$ manifold $M$ 
of dimension $d\ge 2$, and
the dimension $d_s$ of its stable bundle is equal to one (this is the situation when
a horocycle flow of Giulietti--Liverani type can  be constructed), then $M$ is homeomorphic
(and thus diffeomorphic) to the torus $\TTT^d$, and $F$ is topologically conjugated to a linear
toral automorphism $A$ of $\TTT^d$ (see \cite{Hiraide}
for a  recent account of this result of Franks and Newhouse).
See Remark~\ref{higherdim} about the limits of our approach
in  dimensions $d\ge 3$. 

\begin{remark}
After this paper was accepted for publication, J\'er\^ome Carrand \cite{Ca},
following a suggestion of Selim Ghazouani, used Denjoy--Koksma  to give
a very short proof of a slightly weaker result (applying to a slightly
larger class of flows): the ergodic integrals of a zero-average observable grow  at most logarithmically.
\end{remark}

%%%%%%%%%%%%%%%%%%%%%%%%%%%%%%

\section{Statement of the main theorem and its two corollaries}
\label{noref}

Fix $r>1$ and let $F: \TTT^2 \to \TTT^2$ be a $C^r$ Anosov diffeomorphism on the two-torus.
This means that there exist $C<\infty$ and
$\lambda_u>1$, $\lambda_s < 1$ such that the tangent space splits as $T \TTT^2=E^s \oplus E^u$,
with $E^s$ and $E^u$ preserved by $DF$, such that, for all $x$,\footnote{$\|\cdot \|$ denotes the norm on  $T_x \TTT^2$ induced by the Riemann metric on $\TTT^2$.}  
$$
\|DF^n_x v\|\le C \lambda_s^n
\|v\|\, ,\,\,   \forall v\in E^s(x)\, , \, \, 
\|DF^{-n}_x v\|\le C \lambda_u^{-n}\|v\|\, , \,\,  \forall v\in E^u(x) \, ,\, \, \,\forall n\ge 1\,  .
$$
By Franks--Newhouse (see e.g. \cite{Hiraide}), since $d_s=1$, the Anosov diffeomorphism $F$ is topologically 
conjugated to a hyperbolic linear toral automorphism. In particular, $F$ is topologically
mixing. In fact, $d_s=1$ also implies that $E^s$
(see \cite[p. 805]{Hiraide}) is orientable. Fixing an orientation
of $E^s$, we assume that $DF$
preserves this orientation.

\subsection{The Artin--Mazur zeta function of $F$}
Let
$A=A_F\in Gl_2( \integer)$ be the hyperbolic matrix topologically conjugated to $F$
(i.e. the unique linear map in the homotopy class of $F$).
The eigenvalues\footnote{We can again reduce to the case $\sigma_A=1$ by considering $F^2$, but we find it instructive to write the zeta function in the general case.}
of $A$ are $$\sigma e^{h_{top}(F)} \mbox{ and } e^{-h_{top}(F)}<1
\, , \mbox{ with } \sigma:=\sigma_A=\det A\in \{\pm 1\} \, .$$
(The contracting  eigenvalue is positive by our  orientation preserving assumption.) Put $\lambda_A=e^{h_{top}(F)}$.
An easy computation gives
that the unweighted (Artin--Mazur) zeta function of $A$ is just
\begin{equation}\label{zzeta}
\zeta_A(z)=\exp \sum_{n=1}^\infty \frac{z^n}{n} \# \Fix A^n
=\begin{cases}
\frac{(1-z)^2}{(1-\lambda_A z) (1-z/\lambda_A)} &\mbox{ if }\sigma =+1\\
\frac{1-z^2}{(1-\lambda_A z) (1+z/\lambda_A)} 
&\mbox{ if } \sigma=-1\, .
\end{cases}
\end{equation}
(Use\footnote{See e.g. the proof of \cite[Lemma 18.6.2]{KH}.} that $\# \Fix A^n=|\det (1-A^{n})|=|1-(\sigma\lambda_A)^n| \cdot (1-\lambda_A^{-n})$ so that 
$\# \Fix A^n=-(-\lambda_A)^n - \lambda_A^{-n}=
\lambda_A^n +(-1)^n \lambda_A^{-n}$ if $\sigma=-1$ and $n$ is odd, while $\# \Fix A^n=-2 + \lambda_A^n+\lambda_A^{-n}$ otherwise,
and handle separately even and odd
$n$ in this case, using 
that $\exp\sum_{j=1}^\infty 2\cdot {z^{2j}}/{2j}=(1-z^2)$.)

\subsection{The extended dynamics $\tilde F$ and the transfer operator
$\widetilde \LL$} 

Identifying $T_x \TTT^2$ with $\real^2$ for any $x\in \TTT^2$, we 
let $x \mapsto \CC^s(x)\subset T_x \TTT^2$ be a smooth stable cone field for $F$, that
is, each $\CC^s(x)$ is a strict subset of $\real^2$ of nonempty interior, with
 $\xi \cdot \CC^s(x)\subset \CC^s(x)$ for all $\xi \in \real$,
and such that 
$$\overline{DF^{-1}( \CC^s(x))} \subset \CC^s(F^{-1}(x))\cup \{0\}\, ,
\, \forall x\in \TTT^2\, , 
$$
and there exists $\nu_s<1$ such that 
$$
\|D F^{-1}_x (v)\|\ge \nu_s ^{-1}\|v\|\, ,\, \forall x \in \TTT^2\, ,
\forall v \in C(F(x))\, .
$$
We consider the following compact subspace of the unit tangent bundle of $\TTT^2$:
$$ \Omega^ s=
\{ (x,v)\in \TTT^2 \times \real^2\mid
\|v\|=1  \, , \quad v \in \overline {\CC^s(x)}\}\, .
$$
Orientability of $E^s$ gives a decomposition $\Omega^ s
=\Omega^ s_+\cup \Omega^ s_-$. Since $DF$ preserves the orientation of $E^s$,
the $C^{r-1}$ map defined on $\{(x,v) \in \Omega^s_+
\mid  DF_x(v)\in \overline{\CC^s(x)}\}$ by
$$\tilde F(x,v)=
\biggl (F(x), \frac{D_x F(v)}{\|D_x F(v)\|}\biggr )
$$
leaves invariant the set
$$
E^s_{1,+}:=\{(x,v)\in \Omega^s_+ \mid v \in E^s(x) \}
=\{(x,v^s_{1,+}(x))\mid x \in \TTT^2 \} \, .
$$

It is easy to see that $(\tilde F,E^s_{1,+})$ is a topologically mixing Axiom A repellor, with stable dimension one and unstable dimension two, and that the expansion in the new unstable direction is  not smaller than $\lambda_u /\nu_s> \lambda_u$. 
For each $n \ge 1$, the periodic orbits of $F^n$ are in bijection with the periodic orbits
of $\tilde F^n$ via the map
$x\mapsto (x, v^s_{1,+}(x))$.
Following \cite{GL}, we set
\begin{equation}\label{deftildeg}
\tilde g(x,v)=\frac{1}
{|\Det DF|_{V}|\circ \tilde F^{-1}(x,v) }\, ,
\end{equation}
where $V(\tilde F^{-1}(x,v))=\real \cdot DF^{-1}_x v \subset T_{F^{-1}(x)} \TTT^2$ is the line generated by $DF^{-1}_x v$,
and we consider the weighted transfer operator $\widetilde \LL$ of $\tilde F$ defined by
\begin{equation}\label{deftildeL}
(\widetilde \LL \tilde \varphi) (x,v)= (\tilde g \tilde \varphi) \circ \tilde F^{-1}(x,v)
\, .
\end{equation}
(As explained in \cite{GL}, the operator $\widetilde \LL$ corresponds to the action of $F$ on one-forms.)

\subsection{The essential spectrum of $\widetilde \LL$ on the Banach space $\widetilde \BB$}
\label{intro}
Since $\lambda_u /\nu_s >\lambda_u$,  and $\tilde d_s=1$, if $r< \infty$ the results\footnote{Theorem~2.1 there shows that one does not need to replace $r-1$ by $r-2$.} of \cite[Theorem 1.1]{BT} (see also
\cite[Chap. 5]{Babook} for a pedestrian account)
 imply that for any $\epsilon >0$,
 there exists a
Banach space $\widetilde \BB=\widetilde \BB^{t,s}(\tilde F)$  of anisotropic distributions on $\Omega^s_+$
(supported in a neighbourhood $\EE$ of $E^s_{1,+}$)
such that the essential spectral radius of $\widetilde \LL$
on $\widetilde \BB$ is strictly\footnote{Take the Banach space
	$C^{t,s}(\tilde F, \EE)$ 
	in \cite[Theorem 1.1]{BT}, for $\epsilon>0$ and  $t-(r-1-\epsilon)<s<0<t$.} smaller than 
\begin{align}\label{deff}
\tilde \rho_{BT}:=\tilde \rho_{BT}(r)&=\exp
\max_{\substack{s,t \in \real\\ t-(r-1-\epsilon)\le s<0<t}} \, \, \, \sup_{\tilde \mu\in  \Erg(\tilde F)}\, 
\{h_{\tilde \mu} + \max\{-t \ell^u_{\tilde \mu}, 
 |s| \ell^s_{\tilde \mu}\}\}\\
\nonumber &= \exp
 \max_{\substack{s,t \in \real\\ t-(r-1-\epsilon)\le s<0<t}} \, \, \, \sup_{ \mu\in  \Erg(F)}\, 
\{h_{ \mu} + \max\{-t \ell^u_{\tilde \mu}, |s| \ell^s_{\mu}\}\, ,
\end{align}
where 
$\ell^{u/s}_{\tilde \mu}=\int \log (\det D F|_{E^{u/s}})\, d{\tilde \mu}$, and
$\Erg(\tilde F)$ is the set
of   ergodic $\tilde F$-invariant probability measures, which is in bijection with
$\Erg(F)$, with  $h_{\tilde \mu}= h_{\mu(\tilde \mu)}$ and
$\ell^{u/s}_{\tilde \mu}=\ell^{u/s}_{\mu(\tilde \mu)}$, where
$\ell^{u/s}_{\mu}=\int \log (\det D F|_{E^{u/s}})\, d{\mu}$. (The $\ell^u$ are strictly positive
and the $\ell^s$ strictly negative.)

Note that  by \cite[Chap.~5]{Babook}  we have $\tilde \rho_{BT}\le 
\frac{e^{h_{top}(F)}}{\lambda^{(r-1)/2}}$ for any $\lambda$ such that 
\begin{equation}\label{deflambda}
1< \lambda <\min  \{\lambda_u, 1/\lambda_s\}\, .
\end{equation}
Since $\tilde F$ is mixing, $\tilde \rho_0:=e^{h_{top}}$ is a simple eigenvalue and the only maximal eigenvalue.
Set 
\begin{equation}\label{deftilderho}
\tilde \rho_{ess}:=\max \{\tilde \rho_{GL}, \tilde \rho_{BT}\}\, .
\end{equation}
The  sets
$\Sigma_{GL}:=\{\tilde \rho \in \spe(\widetilde \LL|_{\widetilde \BB_{GL}})  \mid |\tilde \rho|> \tilde \rho_{ess}\}$
and $\Sigma_{BT}:=\{\tilde \rho \in \spe(\widetilde \LL|_{\widetilde \BB})  \mid |\tilde \rho|> \tilde \rho_{ess}\}$
coincide (including multiplicities) by \cite[App. A.2]{BT}. In addition, the finitely many corresponding generalised eigenvectors lie in $\widetilde \BB_{GL}\cap \widetilde \BB$.
In view also of the results about the dynamical
determinant \eqref{defdin} recalled below, it is  legitimate to call the 
eigenvalues\footnote{The eigenvalues  $\{\tilde \rho_j\}_{j=1}^N$ are in bijection with
the possible poles $\upsilon_j$ of the Fourier transform of the correlation
function of the measure of maximal entropy of
$F$,  see also Corollary~\ref{cor2}.}  $\Sigma_{BT}=\{\tilde \rho_j\}_{j=0}^{N}$ 
(repeated with multiplicity, ordered with decreasing modulus) of $\widetilde \LL$ of modulus 
$|\tilde \rho_j|>\tilde \rho_{ess}$
the {\it resonances} of $F$. We call\footnote{Note that $N_1=N_{1,GL}$.} the $N_1\le N$ resonances $\{\tilde \rho_j\}_{j=1}^{N_1}$ with 
 $1\le |\tilde \rho_j | <e^{h_{top}(F)}$ the {\it deviation resonances} of $F$.
 
 Since $\sup_{\mu \in \Erg(F)}\, 
\{h_{\mu} -\tau  \ell^u_{\mu}\}<0$ if $\tau >1$, it is not hard to see that
$\tilde \rho_{BT}<1$ if 
\begin{equation}\label{??}
r> r_1 :=1+\epsilon+ \sup_{{\mu} \in \Erg(F)}
\frac{\ell^u_{\mu}-\ell^s_{\mu}}{-\ell^s_{ \mu}}\, ,
\, .
\end{equation}
(If $F$ preserves  area, \eqref{??} reduces to $r_1=3+\epsilon$, taking
$t$ and $-s$ arbitrarily close to $(r-1)/2$.)
Note that \eqref{??} implies
\begin{equation}\label{???}r_1\le 1+\epsilon+\frac{-\log \Lambda_s+\log \Lambda_u}{-\log \lambda_s}\, ,
\end{equation}
with $\Lambda_u\ge \lambda_u$ and $\Lambda_s\le \lambda_s$ the strongest expansion,
 respectively  contraction
of $F$.

 \subsection{The determinant $\tilde d(z)$ of $\widetilde \LL$.}

Theorems 1.5 and \S2 of \cite{BT}  (see also \cite[Chap. 6]{Babook}) give  that the dynamical determinant
\begin{equation}\label{defdin}
d_{\tilde F, \tilde g}(z)=\exp -\sum_{n=1}^\infty \frac{z^n}{n} 
\sum_{(x ,v)\in \Fix \tilde F^n}
\frac{\prod_{k=0}^{n-1} \tilde g(\tilde F^k (x,v) )}{|\det (\Id - D\tilde F^{-n}(x, v))}
\end{equation}
admits a holomorphic extension to the open disc  of radius $\tilde \rho_{BT}^{-1}$,
in which the zeroes of $d_{\tilde F, \tilde g}(z)$  are exactly
the inverses of the eigenvalues of $\widetilde \LL$ of modulus $>\tilde \rho_{BT}$ 
(the order of the zero coincides with the algebraic multiplicity of the eigenvalue).

\subsection{Statement of results}

We now  state  our main result, using the notation from \S\ref{intro}, in particular
\eqref{deftildeL}, \eqref{deff},  and \eqref{defdin}:

\begin{theorem}[Absence of deviation resonances]\label{nores}
Fix $r>1$ and let $F$ be a $C^r$ Anosov diffeomorphism of the two-torus preserving
the orientation of $E^s$. 
	
The only zero of the dynamical determinant $d_{\tilde F, \tilde g}(z)$ 
in the closed  disc of radius $\min \{1,\tilde \rho_{BT}^{-1}\}$  is
a simple zero at $e^{-h_{top}(F)}$.
In particular the spectrum of the operator $\widetilde \LL$   (acting on $\widetilde \BB$)
outside of the open  disc of radius $\max \{1,\tilde \rho_{BT}\}$ consists 
in a simple eigenvalue at $e^{h_{top}(F)}$. (In the notation of \S\ref{intro}, we have $N_1=0$.)
\end{theorem}

\begin{remark}If $F$ does not preserve the orientation of $E^s$, the same
result holds up  to
introducing  a non-mixing extension $\tilde F$ of $F$ such that
 $\tilde F^{-1}$ exchanges $\Omega^ s_+$ and $\Omega^ s_-$.
 The only difference is that there are two maximal eigenvalues, $\pm e^{h_{top}(F)}$.
\end{remark}
	
\noindent Recalling $r_0$ from \S\ref{sec0} (see also Remark~\ref{rr0}), 
and using the notation from \S\ref{intro},
in particular \eqref{deftilderho} and \eqref{??},  we get:
	
\begin{corollary}[No deviations for horocycle ergodic integrals]\label{cor1}
Let $F:\TTT^2\to \TTT^2$ be a $C^r$ Anosov diffeomorphism  which preserves the
orientation of $E^s$. If $r\ge \max\{r_0, r_1\}$, so that $\tilde \rho_{ess}<1$, 
%(recall \eqref{deftilderho}) 
then  for any
$\delta>0$
there\footnote{$\|f\|_{C^{r-1}}$ can be replaced by a weaker norm.  In the area
preserving case, $C^{(r-1)/2}$.} exists  $C'<\infty$ such that
for any $C^{r-1}$ function $f:\TTT^2\to \complex$
\begin{equation}\label{GLexp''}
\sup_x |H_{x,T}(f)- T  \cdot \mu^s(f) |\le  
C' (T^{\theta'_{min}} \|f\|_{C^{r-1}}+ \sup |f|)\, , \, \forall \, T >0 \, ,
\end{equation}
where  (taking small enough $\delta$)
$$\theta'_{min}= \frac{\log \tilde \rho_{ess}+\delta}{h_{top}}<0 \mbox{ if } N=0\, ,
\,\,\,
\theta'_{min}=\frac{\delta+ \log \max \{\tilde \rho_{j}\}_{j=1}^N }{h_{top}}<0
\mbox{ if } N\ge 1\, .
$$	
In particular $\mu^s(f)= 0$ if and only if $f$ is a continuous coboundary.
\end{corollary}

 \begin{remark}\label{add}Our result does not exclude the existence of
 obstructions to Lipschitz (or  H\"older) regularity of the solution $\bar f$
	of the cohomological equation for $f$ when $\mu^s(f)=0$ \cite[end of \S 5.1.2, Remark 5.10]{GL}. Such constructions involve a different transfer
	operator (see the proof of \cite[Theorem 2.12]{GL}).\end{remark}

\begin{proof}[Proof of Corollary~\ref{cor1}]
Since we assume $\tilde \rho_{ess}<1$, the corollary follows from the expansion \eqref{GLexp} from
\cite[Theorem 2.8]{GL} combined with Theorem~\ref{nores} and \eqref{dominant}.
(If  $N\ge 1$,  taking
 $\delta\in (0, -\log \max \{\tilde \rho_{j}\}_{j=1}^{N})$ handles the possible
Jordan blocks. This could be replaced by an appropriate
power of $\log T$ in front of  $T^{\theta'_{min}}$.)
\end{proof}

Our final result is about  the unique measure of maximal entropy
$\mu_{top}$ of $F$:
	
\begin{corollary}[Rates of mixing for the measure of maximal entropy]\label{cor2}
Let $r>1$ and let $F:\TTT^2\to \TTT^2$ be a $C^r$ Anosov diffeomorphism on the two-torus. If 
 $\tilde \rho_{BT}<1$, then there exist
$C<\infty$ and $\rho_{top}<e^{-h_{top}(F)}$ such that for any $C^{r-1}$ functions\footnote{The $C^{r-1}$ norms can be replaced by weaker norms. In the area
preserving case, $C^{(r-1)/2}$.}
 $f_1, f_2:\TTT^2\to \complex$$$
 |\int( f_1 \circ F^k) f_2 \, d\mu_{top}-\int f_1 \, d\mu_{top}
 \int f_2 \, d\mu_{top}|\le C \rho_{top}^k \|f_1\|_{C^{r-1}} \|f_2\|_{C^{r-1}} \, ,
 \, \, \forall k \ge 0 \, .
$$
If 
 $\tilde \rho_{BT}\ge 1$, then for any $\rho_{top}>e^{-h_{top}(F)}\tilde \rho_{BT}$ there exists $C$
 such that the above bound holds. (Note that 
 $e^{-h_{top}(F)}\tilde \rho_{BT}<1$.)
\end{corollary}

Recalling \eqref{??}, note that if $r\ge r_1$ then $\tilde \rho_{BT}<1$. Clearly, if $\tilde \rho_{BT}\ge 1$, then the bound
$\rho_{top}$ given by Corollary~ \ref{cor2} is $\ge e^{-h_{top}(F)}$.
	
\begin{remark}[Comparing with the results of \cite{F}]\label{compareGF}
In our setting, Forni's results imply in particular\footnote{We do not discuss Forni's
analog of Corollary~\ref{cor1}, which contains an additional logarithmic factor
$(\log T)^2$ in the error term.} that if $F$ is $C^r$ for $r>2$, there exists
$C<\infty$  such that for any $C^1$ functions $f_1$ and $f_2$, we have
\begin{align}
\label{gio}	& |\int( f_1 \circ F^k) f_2 \, d\mu_{top}
	 - \int f_1 \, d\mu_{top}  \int f_2 \, d\mu_{top}| \le   C k^2 e^{- k h_{top}}\|f_1\|_{C^{1}} \|f_2\|_{C^{1}}, \,\, \forall k \ge 0   \, .
\end{align}	
It suffices to assume that $r>2$ and $f_1$ and $f_2$ are $C^1$ to get \eqref{gio}, but the error term
there  is not as good as
in Corollary~\ref{cor2}.
\end{remark}

\begin{proof}[Proof of Corollary~\ref{cor2}]
We may assume that $F$ preserves the orientation of $E^s$,
since otherwise we can replace
$F$ by $F^2$ (the two maps have the same measure of maximal
entropy and $h_{top}(F^2)=2 h_{top}(F)$).

Let    $\mu^u=\pi^*(\tilde \mu^u)$ where $\widetilde \LL(\tilde \mu^u)=e^{h_{top}(F)} \tilde \mu^u$, 
normalised by
$\mu^s(\mu^u)=\tilde \mu^s(\tilde \mu^u)=1$. By \cite{GL2},  
the distribution $\varphi \mapsto \mu^s(\varphi \mu^u)$
is in fact the unique measure of maximal entropy $\mu_{top}$ of $F$.
(The fact that \cite{GL2} use another Banach space
does not matter by \cite[App. A.2]{BT}.)
Then using the spectral decomposition for $\widetilde \LL$ and the 
information from Theorem~\ref{nores} gives the claim, just like in \cite{GL2}.
(For example, if $\tilde \rho_{BT}<1$, we may choose
$\rho_{top}$ as follows: If $\widetilde \LL$ has no eigenvalue except  $e^{h_{top}(F)}$ of modulus
strictly larger than $\tilde \rho_{BT}$, then we can take any $\rho_{top}>\tilde \rho_{BT}/e^{h_{top}(F)}$
since $\tilde \rho_{BT} <1$.
Otherwise, letting $ \tilde \rho_1$ be the eigenvalue 
of $\widetilde \LL$ of  largest modulus $\ne e^{h_{top}(F)}$, 
we have  $|\tilde \rho_1|<1$ by Theorem~\ref{nores}, and 
for any fixed integer $q\ge 0$, taking 
 $\rho_{top}>| \tilde \rho_1|/e^{h_{top}(F)}$,
 we have $k^q| \tilde \rho_1|^k e^{-k h_{top}(F)}
<C \rho_{top}^k$ for all $k$.)
\end{proof}

\section{Proof of Theorem~ \ref{nores}}

In order to prove Theorem~ \ref{nores}, we
introduce the following  notation:
If $\Phi:M\to M$ is an Axiom A diffeomorphism with  a basic set $X$
 and $\phi: X\to \complex$ admits a
H\"older extension to a  neighbourhood of $X$, we set $\phi^{(n)}=\prod_{k=0}^{n-1}  (\phi\circ \Phi^k)$, and 
$$
 d_{\Phi, \phi}(z)=\exp -\sum_{n=1}^\infty \frac{z^n}{n} 
\sum_{y \in \Fix \Phi^n \cap X}
\frac{  \phi^{(n)}(y) }{|\det (\Id - D\Phi^{-n}(y))} \, .
$$
For such  a fixed $\Phi$ and $X$, we say that $\phi_1 <_{Per} \phi_2$ if $\phi_1, \phi_2: X\to \real_*^+$  satisfy
$$
\limsup_{n \ge 1}\,  \sup_{y \in \Fix \Phi^n \cap X}\, 
\frac{(\phi_1^{(n)}(y))^{1/n}}{(\phi_2^{(n)}(y))^{1/n}} < 1 \, .
$$

Finally, 
the following two elementary facts will be used in the proofs: For a $2 \times 2$ matrix $Q$,
we have
$
\Det (\Id -Q)= 1+ \Det Q - \tr Q
$.
If $\widetilde Q$ is a $3\times 3$ matrix of the form
$\widetilde Q =\left (\begin{smallmatrix}Q & q\\ 0 & \eta \end{smallmatrix}\right )$,
with $q$ an arbitrary column vector in $\real^2$ and $\eta \in \real$, we have
$$
\Det (\Id -\widetilde Q)= (1+ \Det   Q - \tr  Q )(1- \eta)  \, .
$$
If $F=A$  (so that $r=\infty$) and  $\sigma_A=1$ then   the above
facts for $Q=A^{-n}$ give\footnote{
Formula \eqref{cute} recovers the eigenvalues of the operator 
associated to $A$ obtained in  \cite[ \S 5.2]{GL}.}
\begin{equation}\label{cute}
 d_{\tilde F, \tilde g }(z)= \prod_{j=0} ^\infty (1-z \lambda_A^{1-2j})\, .
\end{equation}
(We shall not use the above expression.) 

\begin{proof}[Proof of Theorem~\ref{nores}]
As a warmup, we assume that $r>3$, we fix $\epsilon >0$ small,
and we consider the transfer operator
\begin{equation}\label{thedef}
\LL \varphi(x)=\frac{\varphi (F^{-1}(x))}{|\Det DF|_{E^s}|\circ F^{-1} (x)}
\end{equation}
acting on the Banach space of distributions $\BB=C^{(1+\alpha)/2,-(1+\alpha)/2}(F, \TTT^2)$ defined in
\cite[\S4]{BT}, where its spectral radius is
 $e^{h_{top}(F)}$, while its
essential spectral radius is strictly smaller  
than (recalling $\lambda>1$ from \eqref{deflambda})
 $$\rho_{ess,\alpha}:=\frac{e^{h_{top}(F)}}{\lambda^{(1+\alpha)/2}}\, , $$ 
where the regularity\footnote{Of course, $\alpha=\infty$ in the linear case $F=A$, but Anosov proved that, generically, $\alpha <1$.}
%If $F$ is $C^3$ and preserves area then \cite{Ha},  by results of 
%Katok--Hurder, $\alpha =O(x \log x)$ (we shall not need this).} 
of the bundle $E^s$ is
$C^{1+\alpha}$ for some $\alpha >0$.
 By mixing, $e^{h_{top}(F)}$ is simple and
 is the only eigenvalue of $\LL$ of maximal modulus. Finally \cite[Theorem 1.5]{BT} also
 implies that  the dynamical determinant $$d_{F,g^s}(z)
 =\exp -\sum_{n=1}^\infty \frac{z^n}{n} \sum_{x\in \Fix  F^n}
\frac{  (g^s)^{(n)}(x) }{|\det (\Id - D F^{-n}( x))|}\, , 
$$
where $g^s=(|\Det DF|_{E^s}|)^{-1}$,
 is holomorphic in the disc of
radius $1/\rho_{ess,\alpha}$ where its zeroes are the inverses of the eigenvalues of $\LL$.

 We will show that 
 $$
 \spec (\LL|_\BB)\cap
 \{\max ( \rho_{ess,\alpha}, 1) < |\rho| \le e^{h_{top}(F)}\}
 =\{ e^{h_{top}(F)}\}\, ,
 $$
 and if, in addition
 $\rho_{ess,\alpha}<1$ (since $\alpha$ is generically smaller
 than $1$, this is a very strong assumption, even for large $r$), then  $\LL:\BB\to \BB$ does not have eigenvalues outside the open unit disc except for $e^{h_{top}}$. The starting point is the fact that the Artin--Mazur zeta function
 of  $F$ coincides with the zeta function \eqref{zzeta} of its linear model:
 \begin{equation}\label{zeta1}
 \zeta_{F}(z)=\exp \sum_{n=1}^\infty \frac{z^n}{n} \bigl (\sum_{x \in \Fix F^n}
 1\bigr ) =
 \zeta_A(z) \, .
 \end{equation}

 Next, we use that  for all $n$ and any $x \in \Fix F^n$,
 \begin{align*}
 1&=
 \frac{|\Det (\Id -DF^{-n}(x))|}{|\Det (\Id -DF^{-n}(x))|}=\frac{|1+ \Det DF^{-n}(x) - \tr (DF^{-n}(x))|}{|\Det (\Id -DF^{-n}(x))|}\\
 &=\frac{(\sigma_A)^n(-1- \Det DF^{-n}(x) + \tr (DF^{-n}(x)))}{|\Det (\Id -DF^{-n}(x))|}\, .
 \end{align*}
 Since for any $x \in \Fix F^n$ we have  (using orientation preserving on $E^s$)
$$ 
\tr (DF^{-n}(x))= (|\Det DF^n|_{E^s}(x)|)^{-1}
+(\Det DF^n|_{E^u}(x))^{-1} \, , 
$$
setting $w=\sigma_A \cdot z$, it follows that 
\begin{equation}\label{magic}
\zeta_A(z)= 
\frac{d_{F,1}(w)\cdot d_{F,1/\Det DF}(w)}{d_{F,g^u}(w)\cdot d_{F,g^s}(w)}
\end{equation}
where $g^u=(\Det DF|_{E^u})^{-1}$.  
Let $\beta>0$ be\footnote{In fact, $\beta=\alpha$ in the present setting.} such that  $E^u(F)$ (and thus $g^u$) is
 $C^{1+\beta}$ and assume that $\epsilon <\beta<2$.
 Then there exists \cite[Theorem 1.1]{BT}  a Banach space $\BB_\beta=C^{(\beta-\epsilon)/2, -(\beta-\epsilon)/2}(F,\TTT^2)$ such that
 the transfer operator  $
\LL_{F, g^u}(\varphi)=
(g^u \varphi )\circ F^{-1}$ acting on $\BB_\beta$ has spectral radius $<1$
(because $|g^u|<_{Per} |\det DF|^{-1}$) and essential
spectral radius strictly smaller than its spectral radius, and, in  addition, 
$d_{F,g^u}(w)$ is holomorphic\footnote{We also have that $d_{F,g^u}(w)$ cannot vanish inside a disc of
radius $>1$. Since this determinant appears in the
denominator, its zeroes do not matter to us.} in a disc of radius $>1$.

We will see below that  the two  determinants  in the numerator
of \eqref{magic}
are holomorphic in the disc of radius $\rho_{SRB,ess}^{-1}>\max\{\rho_{ess,\alpha}^{-1} ,1\}$.
Since $d_{F,g^s}(w)$ is the dynamical determinant of $\LL$
which  is holomorphic in the disc of
radius $1/\rho_{ess, \alpha}$, the domains\footnote{Note that $d_{F,g^u}(z)$ is holomorphic
 in a disc of radius $>1$ with no zeroes in the
closed unit disc even if $\rho_{ess,\alpha} >1$. The argument for $\widetilde \LL$
will exploit a similar feature.} of holomorphy 
(and of spectral interpretation of zeroes) of the four determinants
in the right hand side of \eqref{magic} include the disc of radius $1/\rho_{ess,\alpha}$.

Applying \cite[Theorem 1.1]{BT} again, the transfer operator $\LL_{F,SRB}(\varphi)=
\frac{\varphi \circ F^{-1}}{|\Det DF|\circ F^{-1}}$ acting on $\BB$ has essential spectral radius 
bounded by $\rho_{SRB,ess} < \rho_{ess,\alpha}$ (we have $\rho_{SRB,ess} \le \lambda^{-(r-1)/2}$) and spectral
radius one, with 
%(by mixing) 
a simple eigenvalue at $1$ as
 only eigenvalue on the unit circle. Therefore, if $\sigma_A=+1$
 then  $d_{F,1/\Det DF}(w)=d_{F,1/|\Det DF|}(z)$
 is holomorphic in the disc of radius  $\rho^{-1}_{SRB,ess}$, with $d_{F,1/|\Det DF|}(1)=0$, and this determinant has no other zeroes in the closed unit disc.
 If $\sigma_A=-1$
 then  $d_{F,1/\Det DF}(w)=d_{F,1/\Det DF}(-z)$
 is holomorphic in the disc of radius  $\rho^{-1}_{SRB,ess}$, with $d_{F,1/|\Det DF|}(-1)=0$, and this determinant has no other zeroes in the closed unit disc.

Finally,  we claim that $d_{F,1}(w)$ is holomorphic in the disc of radius  $\rho^{-1}_{SRB^*,ess}>1$,
with $d_{F,1}(1)=0$, where $1$  is  simple zero and the only zero
of $d_{F,1}(w)$ in the closed unit disc:
Indeed,  $d_{F,1}(w)$ can be viewed as the determinant
of the operator $$\LL_{F^{-1}, SRB} \varphi =|\Det DF | \cdot (\varphi \circ F)$$
acting on the Banach space $\hat \BB=C^{(t, s}(F^{-1})(F^{-1}, \TTT^2)$, 
for suitable $t-(r-1)<s<0$, associated to $F^{-1}$,
on which its essential spectral radius is bounded by $\rho_{SRB^*,ess}< \tilde \rho_{ess,\alpha}$. The spectral radius 
of
$\LL_{F^{-1}, SRB} $
is equal to $1$, and, since $F^{-1}$ is mixing, the eigenvalue
$1$ is simple  and it is the only eigenvalue of
modulus equal to one. (The  respective simple zeroes of $d_{F,1}(z)$ and
$d_{F,1/|\Det DF|}$ at $z=1$  account for the double zero of
$\zeta_{F}(z)$ at $z=1$ when $\sigma_A=+1$. Otherwise, the simple zeroes at $z=\pm 1$ account
for the zeroes of $(1-z^2)$.)

Assume by 
contradiction that $\LL$ 
has an eigenvalue $\rho\ne e^{h_{top}(F)}$ with
$|\rho|\ge 1$, then  $d_{F,g}(1/\rho)=0$. This would imply that 
$d_{F,1}(1/\rho)=0$ or $d_{F,1/|\Det DF|}(1/\rho)=0$ (to get a cancellation), and both
claims are impossible.

\medskip
We  now  prove the theorem: 
The only differences are that we will need $3\times 3$
matrices instead of $2\times 2$ matrices, and six determinants instead of four.
Also we only need that $E^u$ and $E^s$ are $C^\gamma$ for some $\gamma>0$.
To fix ideas, assume first that $\tilde \rho_{BT}<1$ so that
$\min \{1, \tilde \rho_{BT}^{-1}\}=1$.
First
 $$
 \zeta_{\tilde F}(z)=\exp \sum_{n=1}^\infty \frac{z^n}{n} \bigl (\sum_{x \in \Fix \tilde F^n}
 1\bigr ) =
 \zeta_A(z) \, .
 $$
 (Recall \eqref{zeta1} and \eqref{zzeta}.)
 Second, at any $\tilde x=(x,v)=(x, v^s_{1,+}(x))\in \Fix \tilde F^n$, we have
 $$D\tilde F^n_{\tilde x} =\left (\begin{smallmatrix}DF^n_x & a\\ 0 & \tilde \lambda^n( x) \end{smallmatrix}\right )
 \mbox { 
 with } \tilde \lambda^n(x)= \frac{\Det DF^n|_{E^u}(x)}{|\Det DF^n|_{E^s}(x)|}\, ,
 $$
  where $|\tilde \lambda^n(x)|\ge
 \frac{ \lambda^n_u} {\lambda^n_s}$. Therefore,
 we may use  for such $\tilde x$ the decomposition
  \begin{align*}
 1&=
 \frac{|\Det (\Id -D\tilde F^{-n}(\tilde x))|}{|\Det (\Id -\tilde DF^{-n}(\tilde x))|}\\
 &=
 \frac
 {(|1+ \Det DF^{-n}(x) - \tr (DF^{-n}(x))|) \bigl (1-\frac{|\Det DF^n|_{E^s}(x)|}{\Det DF^n|_{E^u}(x)} \bigr )}
 {|\Det (\Id -D\tilde F^{-n}(\tilde x))|}\, . 
 \end{align*}
Third, since for all $n$ and all $x\in \Fix F^n$,
\begin{align*}
&(|1+ \Det DF^{-n}(x) - \tr (DF^{-n}(x))|\\
&\quad =(\sigma_A)^n(-1- \Det DF^{-n}(x) +  (|\Det DF^n|_{E^s}(x)|)^{-1} +(\Det DF^n|_{E^u}(x))^{-1})
\, ,
\end{align*}
setting $w=\sigma_A \cdot z$, it follows that 
\begin{align}
\nonumber \zeta_A(z)
&=
\frac{d_{\tilde F,1}(w)\cdot d_{\tilde F,1/\Det DF}(w)
\cdot d_{\tilde F,\tilde g^u}(w)\cdot d_{\tilde F,(\tilde g^u)^2/g^s}(w)}
{d_{\tilde F,g^s}(w)\cdot d_{\tilde F,\tilde g^u}(w)
\cdot d_{\tilde F,\tilde g^u/\tilde g}(w) \cdot d_{\tilde F,(\tilde g^u)^2}(w)}\\
\label{magi2}& =
\frac{d_{\tilde F,1}(w)\cdot d_{\tilde F,1,\Det DF}(w)
\cdot d_{\tilde F,(\tilde g^u)^2/\tilde g}(w)}
{d_{\tilde F,\tilde g}(w)\cdot d_{\tilde F,\tilde g^u/\tilde g}(w) \cdot d_{\tilde F,(\tilde g^u)^2}(w)}\, , 
\end{align}
where $\tilde g$ was defined\footnote{Note that $\tilde g(x,v^s_{1,+}(x))=g^s(x)$.} in \eqref{deftildeg} and $\tilde g^u(x, v):=g^u(x)$
is $C^\gamma$.

As mentioned
above,
$d_{\tilde F,\tilde g}(z)$ is the dynamical determinant of $\widetilde \LL$ acting
on $\widetilde \BB$,
which is holomorphic in the disc of radius $1/\tilde \rho_{BT}$.

It is easy to see that $d_{\tilde F,1}(w)$  is the dynamical determinant
of the transfer operator associated to the SRB measure of the mixing
attractor $\tilde F^{-1}$
acting on the Banach space $\widetilde \BB_{\tilde F^{-1}}=\CC^{t, s}(\tilde F^{-1},\EE)$
for suitable $t-(r-1)<s<0$,
associated to  $\tilde F^{-1}$, so, by \cite[Theorem 1.5 and \S2]{BT}, this factor is holomorphic in
a disc of radius $\tilde \rho_{SRB^*,ess}^{-1}>\tilde \rho_{BT}^{-1}>1$, and its only zero in the closed unit disc is a simple zero at $w=1$. (If $\sigma_A=-1$ then this gives a zero at $z=-1$.)

Next, $d_{\tilde F,1/|\Det DF|}(w)$  is the dynamical determinant
of the transfer operator $\MM$ of the mixing
repellor $\tilde F$ weighted by $1/|\Det DF|$,
acting on the Banach space $\widetilde \BB$. By the Pesin
entropy formula the pressure\footnote{The pressure of 
	$-\log |\Det D \tilde F|_{E^u}|$
	for the repellor  $\tilde F$ is strictly smaller than zero.} of 
$-\log |\Det DF|_{E^u}|$
for the attractor  $F$ is equal to zero. Since this pressure coincides with the pressure
of 
$-\log |\Det DF|_{E^u}|$
for the repellor  $\tilde F$,
applying \cite[Theorem 1.5]{BT} the determinant $d_{\tilde F,1/|\Det DF|}(w)$ is holomorphic in
a disc of radius $\tilde \rho_{SRB,ess}^{-1}>\tilde \rho_{BT}^{-1}>1$ and its only zero in the closed unit disc is a simple zero at $w=1$.  
(The fixed point of $\MM$ is the SRB measure of $F$
in  $x$  multiplied by a Dirac mass at 
$v^1_{x,+}$.) If $\sigma_A=-1$, replacing $\MM$ by $-\MM$, we get that
the determinant $d_{\tilde F,1/\Det DF}(w)$ is holomorphic in
a disc of radius $\tilde \rho_{SRB,ess}^{-1}>\tilde \rho_{BT}^{-1}>1$ and its only zero in the closed unit disc is a simple zero at $w=-1$, giving a zero at $z=1$.

The last determinant in the numerator of \eqref{magi2}
is $d_{\tilde F,(\tilde g^u)^2/\tilde g}(w)$, which
is associated to a transfer operator $\NN$ of spectral radius $<1$
 on the Banach space 
$\widetilde \BB_{\gamma}=C^{\gamma/2, -\gamma/2}(\tilde F, \EE)$ using $(\tilde g^u)^2/\tilde g(z)<_{Per} |\det DF|^{-1}$
 (the essential spectral  radius of $\NN$ is  $<1$, but we do not claim
 that this essential spectral  radius  is smaller than $\tilde \rho_{BT}$). Therefore,
$d_{\tilde F,(\tilde g^u)^2/\tilde g}(w)$ cannot
vanish on the closed unit disc.

Finally, the spectral radii of the two remaining operators in the denominator
of \eqref{magi2}  acting
on $\widetilde \BB_{\gamma}$ are  smaller
 than $1$ since
 $$
{|\tilde g^u|}/{\tilde g}
<_{Per} {|\det DF|}^{-1} \, , \qquad
(\tilde g^u)^2 
<_{Per} {|\det DF|}^{-1}  \, .
$$
%(For example, $d_{\tilde F,g^u}(z)$ is the dynamical determinant of $\NN_1(\varphi)=
%(g^u \varphi )\circ \tilde F^{-1}$, which is holomorphic in the disc of radius $1/\tilde \rho_{ess}>1$
%and does not vanish in the closed unit disc since the spectral radius
%of $\NN_1$ is strictly smaller than $1$.)
 
 So the right-hand-side of \eqref{magi2} is a quotient of two holomorphic functions in a disc of radius
 $>1$, such that the only zeroes of the numerator in the closed unit disc are
 a double zero at $z=1$ when $\sigma_A=1$, and simple zeroes at 
 $z=\pm 1$ otherwise.
 It follows that the only possible zero of $d_{\tilde F,\tilde g}(z)$  in the closed unit
 disc is a simple zero at $\lambda_A^{-1}=e^{-h_{top}(F)}<1$. This ends the proof of the theorem if $\tilde \rho_{BT} <1$.
 
 If $\tilde \rho_{BT} \ge 1$, the same argument works, replacing
 the unit disc by the disc of radius $\tilde \rho_{BT}$ for spectral claims and by the disc of radius
 $\tilde \rho_{BT}^{-1}$ for determinants. 
\end{proof}

\begin{remark}[Higher dimension]\label{higherdim}
Consider an Anosov diffeomorphism $F$ with $d_s=1$ and $d_u=2$. Then, assuming that 
$DF$ preserves the orientation of $E^s$, and that the eigenvalues
of the hyperbolic linear matrix $A$ conjugated to $F$ satisfy
$\lambda_s<1<\lambda_{u,min}<\lambda_{u,max}$,
the unweighted Artin--Mazur zeta function of $F$ is (note that $\lambda_s\lambda_{u,min}\lambda_{u,max}=1$, and, since $d$ is odd,
 $+1$ and $(-1)^d \det (A^{-n})$ cancel for all $n$)
$$
\zeta_A(z)=\frac{(1-\lambda_s z)(1-\lambda_{u,min}z)(1-\lambda_{u,max}z)}
{(1-z/\lambda_s)(1-z/\lambda_{u,min})(1-z/\lambda_{u,max})} \, .
$$
Factoring the zeta function as a product of  dynamical determinants 
for an extended dynamics  $\tilde F$ as in the proof of Theorem~\ref{nores},
one cannot  exclude a priori eigenvalues of modulus
$\ge 1$ for the operator (in the denominator) weighted by $1/|\det DF|_{E^s}|$, since the
corresponding zeroes could  be cancelled by the determinants
of the transfer operators in the numerator weighted
by $\det DF|_{E^u_{max}}$ or even  $\det DF|_{E^u_{min}}$, which
also  have spectral radius strictly larger than $1$. 
\end{remark}

\appendix
\section{Properties of Giulietti--Liverani flows}\label{appA}

Let $h^t$ be a (not necessarily unit speed) Giulietti--Liverani flow of a $C^r$ Anosov diffeomorphism $F$.
If $r>2$, since the stable bundle of $F$ is $C^{1+\alpha}$
for some    $\alpha >0$ (see e.g. \cite{Ha} and references therein), the map $x \mapsto h^t(x)$
is $C^{1+\alpha}$. Clearly, $h^t$ cannot have fixed points or periodic orbits.
It is well known that a periodic orbit-free flow on $\TTT^2$ has a global transversal,
with the corresponding Poincar\'e map  topologically conjugated
to an irrational rigid rotation $R_\omega$, and that the flow is uniquely ergodic  (and thus minimal, using the conjugacy with the linear model).  
By \cite[Lemma 1.1]{GL}, the rotation number $\omega$ of the rotation thus associated to a
 Giulietti--Liverani flow  satisfies\footnote{As pointed out by the Editor, in the first statement of  \cite[Lemma 1.1]{GL}
 	the sentence ``it is topologically conjugated to a rigid rotation with rotation number $\omega$ such that...'' should be
 	replaced by ``it is topologically orbit equivalent with a flow whose
 	 Poincar\'e map on a global transversal has rotation number $\omega$ such that...''}
\begin{equation}\label{foro}
b \omega^2 + (a-d) \omega - c=0\, ,
\end{equation}
where $A= \left (\begin{smallmatrix}a & b\\ c & d \end{smallmatrix}\right )\in PSL(2,\integer)$ is the hyperbolic linear
toral automorphism topologically conjugated to $F$.
In particular, $b\ne 0$, and the continued fraction of $\omega$ is periodic and thus of constant type. (The orientation preserving 
assumption also implies that the contracting eigenvalue of $A$  is positive.)

Finally, by \cite[Lemma 1.1]{GL}, if $r \ge 3$, any $C^{r}$ periodic orbit-free flow $h_0^t$ on $\TTT^2$ whose Poincar\'e map $P_0$ has rotation number $\omega$ satisfying \eqref{foro} for some hyperbolic
matrix $A\in \left (\begin{smallmatrix}a & b\\ c & d \end{smallmatrix}\right )\in PSL(2,\integer)$, with positive contracting
eigenvalue, 
is the (non necessarily unit speed) Giulietti--Liverani flow of an Anosov diffeomorphism $F_0$, which is $C^{r'}$ for all $r' <r-1$:
Indeed, 
the $C^r$ circle diffeomorphism $P_0$ has rotation number $\omega$ of constant type and is thus  conjugated with the rigid rotation $R_\omega$ via
a $C^{r'}$ circle diffeomorphism $\phi$ \cite[Th\'eor\`eme fondamental, p. 9]{Herman}, for any $r'<r-1$.
Let $h_1^t$ be the unit speed $C^r$ reparametrisation of $h_0^t$.
Then the (unit speed) linear flow $h_\omega^t$ over $R_\omega$
 is the Giulietti--Liverani flow of the hyperbolic automorphism $A$.
In addition, the flow $h_\omega^t$ is conjugated with $h_1^t$ via a  $C^{r'}$ toral diffeomorphism $\Phi$ which coincides with
$\phi$ on the transversal and maps the stable lines of $h_\omega^t$
to the orbits of $h_i^t$,
$i=0, 1$. Finally, $h_1^t$ is the
(unit speed)  Giulietti--Liverani flow of the
$C^{r'}$ Anosov diffeomorphism $F_0= \Phi \circ A \circ \Phi^{-1}$.
(Use that the stable manifolds of $F_0$ coincide with the orbits of $h_i^t$,
$i=0, 1$.)

%%%%%%%%%%%%%%

\end{document}